\documentclass{amsart}

\usepackage[utf8]{inputenc}

\usepackage{amsfonts}
\usepackage{amsmath}
\usepackage{amssymb}
\usepackage{mathtools}
\usepackage{tabularx}

\usepackage{geometry} 
\usepackage{pxfonts}
\usepackage{euscript}
\usepackage{bbold,bbm}
\usepackage{comment}
\usepackage{scalerel}
\usepackage{xcolor}
\usepackage{tikz-cd}

\usepackage{hyperref}
\usepackage{graphics}
\usepackage{epstopdf} 
\usepackage[all,2cell,arrow,matrix,tips]{xy} \UseAllTwocells \SilentMatrices
\usepackage{graphicx}
\usepackage{verbatim}
\usepackage{leftidx}
\usepackage{enumerate}
\usepackage{phonetic}
\usepackage{lipsum}
\usepackage{subfigure}
\usepackage{float}

\usepackage[backend=biber]{biblatex}
\theoremstyle{definition}

\newtheorem{thm}{Theorem}[section]
\newtheorem{prop}[thm]{Proposition}

\newtheorem{lem}[thm]{Lemma}
\newtheorem{rem}[thm]{Remark}

\newcommand\Rb {\mathbb{R}}
\newcommand\Cb {\mathbb{C}}

\newcommand\CE {\EuScript{E}}

\newcommand\CH {\EuScript{H}}

\newcommand{\Cla} {\mathcal{A}}
\newcommand{\Clb} {\mathcal{B}}

\newcommand{\Cll} {\mathcal{L}}
\newcommand{\Clm} {\mathcal{M}}

\newcommand{\Clw} {\mathcal{W}}

\newcommand{\bv} {\mathbf{v}}

\DeclareMathOperator{\id}{id}

\DeclareMathOperator{\diag}{diag}

\DeclareMathOperator{\tr}{tr}

\DeclareMathOperator{\re}{Re}
\DeclareMathOperator{\im}{Im}

\DeclareMathOperator{\Alg}{Alg}
\DeclareMathOperator{\Lat}{Lat}
\DeclarePairedDelimiterX\braket[2]{\langle}{\rangle}{#1 \delimsize\vert #2}

\begin{document}

\title[]{The 2-Norm Distance on the Reflexive Lattice Generated by a Double Triangle Lattice of Projections}
    
\author{Wei Yuan}
\address{Institute of Mathematics, Academy of Mathematics and Systems Science, Chinese Academy of Sciences, Beijing, 100190, China; School of Mathematical Sciences, University of Chinese Academy of Sciences, Beijing 100049, China.}
\email{wyuan@math.ac.cn}
\thanks{Research of the author (W. Yuan) is supported by the NSFC under grant numbers 12471124 and 12571130.}


\date{}

\begin{abstract}
Using operator-valued subordination, we derive a formula for the $2$-norm distance on reflexive lattices generated by double triangles of projections under freeness assumptions on their graph coefficients, and provide explicit evaluations in both the scalar and free-projection cases.
\end{abstract}

\maketitle

\section{Introduction and Preliminaries}

A family of orthogonal projections $\Cll \subset \Clb(\CH)$ is called reflexive if $\Cll = \Lat \Alg \Cll$, where $\Alg \Cll$ denotes the algebra of bounded linear operators that leave invariant the range of every projection $P \in \Cll$, and $\Lat \Alg \Cll$ denotes the family of orthogonal projections onto the invariant subspaces of $\Alg \Cll$. For any set of projections $\Cll$, the lattice $\Lat \Alg \Cll$ is reflexive. 

A double triangle is a five-element projection lattice $\{0, I, P_1, P_2, P_3\}$ in which the three non-trivial projections satisfy the pairwise meet and join relations $P_i \wedge P_j = 0$ and $P_i \vee P_j = I$ for $i \neq j$. In this paper, we study the $2$-norm distance on the reflexive lattice generated by a double triangle of projections satisfying certain freeness conditions in a finite von Neumann algebra.

Throughout the paper, let $\Cla \subset \Clb(\CH)$ be a type II$_1$ factor acting on a complex Hilbert space $\CH$, and let $\Clm = \Cla \otimes M_2(\Cb)$. We denote by $\tau$ the unique faithful normal tracial state on $\Cla$, and define the tracial state $\tau_2 = \tau \otimes \tr$ on $\Clm$, where $\tr$ is the normalized trace on $M_2(\Cb)$. Let $L^2(\Clm, \tau_2)$ denote the completion of $\Clm$ with respect to the $2$-norm $\|A\|_2 = \tau_2(A^*A)^{1/2}$ for $A \in \Clm$. Since $\Clm$ (resp. $\Cla$) is a finite von Neumann algebra, the collection of closed, densely defined operators affiliated with $\Clm$ (resp. $\Cla$) forms a $*$-algebra under the strong sum $A \mathbin{\hat{+}} B$ and strong product $A \mathbin{\hat{\cdot}} B$. In what follows, we suppress the explicit closure notation and write sums and products as if they were operations on ordinary bounded operators.

Let 
\begin{align*}
    Q(\infty) &= 
    \begin{pmatrix}
        I & 0 \\
        0 & 0
    \end{pmatrix}, \\
    Q(0) &= 
    \begin{pmatrix}
        K_1 & \sqrt{K_1(I-K_1)}V_1 \\
        V_1^*\sqrt{K_1(I-K_1)} & V_1^* (I-K_1)V_1
    \end{pmatrix}, \\
    Q(-1) &= 
    \begin{pmatrix}
        K_2 & \sqrt{K_2(I-K_2)}V_2 \\
        V_2^*\sqrt{K_2(I-K_2)} & V_2^* (I-K_2)V_2
    \end{pmatrix}
\end{align*}
be projections in $\Clm$ generating a double triangle lattice, i.e.,
\begin{equation}\label{equ:double_tri}
    Q(z) \wedge Q(w) = 0 \quad \text{and} \quad Q(z) \vee Q(w) = I
\end{equation}
for all distinct $z \neq w$ in $\{0, -1, \infty\}$, where $K_1, K_2 \in \Cla$ are positive contractions satisfying $\ker(I-K_i) = \{0\}$, and $V_1, V_2 \in \Cla$ are unitaries. By Lemma 2.2 in \cite{HY12}, the condition \eqref{equ:double_tri} is equivalent to
\begin{align}\label{equ:double_tri_equv}
    \ker \left (\sqrt{K_1(I-K_1)^{-1}}V_1 - \sqrt{K_2 (I-K_2)^{-1}}V_2 \right)  = 0.
\end{align} 

It was shown in \cite{HY12} that, with $\{0, I\}$ removed, the reflexive lattice 
\begin{align*}
 \Cll:= \Lat \Alg \{Q(\infty), Q(0), Q(-1)\}   
\end{align*}
is homeomorphic to the two-dimensional sphere. More precisely, this homeomorphism identifies $\widehat{\Cb} = \Cb \cup \{\infty\} \cong S^2$ by sending $\infty$ to $Q(\infty)$ and each $z \in \Cb$ to the orthogonal projection $Q(z)$ onto the graph of the densely defined closed operator
\begin{align}\label{equ:Tz_formula}
    T(z) := (1+z)\sqrt{K_1(I-K_1)^{-1}}V_1 - z\sqrt{K_2(I-K_2)^{-1}}V_2,
\end{align}
which is given explicitly by
\begin{align*}
Q(z) = 
\begin{pmatrix}
    I - (I + T(z)T(z)^*)^{-1} & T(z)(I + T(z)^*T(z))^{-1} \\
    (I + T(z)^*T(z))^{-1}T(z)^* & (I + T(z)^*T(z))^{-1}
\end{pmatrix}.
\end{align*}
Moreover, any three distinct nontrivial projections in the lattice determine the same reflexive lattice.

In \cite{Y17}, Wu and Yuan studied the symmetries of the lattice $\Cll$ induced by $*$-automorphisms of the ambient finite factor generated by $\{Q(\infty), Q(0), Q(-1)\}$, proving that the lattice-preserving automorphism group is isomorphic, as a topological group, to a closed subgroup of $\mathrm{SO}(3)$. When $Q(\infty), Q(0)$, and $Q(-1)$ are freely independent, by computing the Fuglede--Kadison determinant of $I - Q(z)Q(w)Q(z)$ for any $z, w \in \widehat{\Cb}$, they determined that the lattice-preserving automorphism group is isomorphic to the symmetric group $S_3$. While the Fuglede--Kadison determinant $\Delta(I - Q(z)Q(w)Q(z))$ captures the geometric mean of the compression $Q(z)(I - Q(w))Q(z)$, the $2$-norm distance
\begin{align*}
  d(z,w) := \|Q(z) - Q(w)\|_2 = \sqrt{2\tau_2\bigl(Q(z)(I - Q(w))Q(z)\bigr)}, \qquad \forall z, w \in \widehat{\Cb},
\end{align*}
is governed by its arithmetic mean. Every symmetry of $\Cll$ induced by a $*$-automorphism of the ambient factor preserves the normalized trace and consequently acts isometrically with respect to this metric. It is therefore natural to seek an explicit description of $d$ in terms of this spherical parametrization. In the free-projection case, adapting the analysis from \cite{Y17} readily yields, for every $z \in \widehat{\mathbb{C}}$,
\begin{align*}
    d(z, 0)^2 &= \frac{|z|}{1 + |z| + |z+1|},\\
    d(z, -1)^2 &= \frac{|z+1|}{1 + |z| + |z+1|},\\
    d(z, \infty)^2 &= \frac{1}{1 + |z| + |z+1|}.
\end{align*}
In particular, 
\begin{align*}
    d(t, 0)^2 &= \frac{1}{2}, \qquad t \in [\infty, -1],\\
    d(t, \infty)^2 &= \frac{1}{2}, \qquad t \in [-1, 0], \\
    d(t, -1)^2 &= \frac{1}{2}, \qquad t \in [0, \infty].
\end{align*}
These constant-distance identities are reminiscent of the boundary arcs of a Reuleaux triangle. The purpose of this paper is to make this geometric picture more precise by deriving an explicit distance formula for arbitrary $z, w \in \Cb \cup \{\infty\}$. Specifically, under a suitable freeness assumption, we establish the following theorem, whose notation will be introduced in the next section.

\begin{thm}\label{thm:gen_formula}
Let $K_j \in \Cla$ be a positive contraction and $V_j \in \Cla$ be a unitary with $\tau(V_j) = 0$, $j = 1, 2$. Assume that $K_1, K_2, V_1, V_2$ are freely independent and satisfy the condition \eqref{equ:double_tri_equv}. For every $z, w \in \Cb$,  
\begin{align*}
    d(z, w) = \left( \frac{1}{\rho(z)} + \frac{1}{\rho(w)} - \frac{\bv^t \Sigma(z,w)\bv}{\rho(z)\rho(w)} \right)^{1/2}
\end{align*}
and 
\begin{align*}
    d(z, \infty) = d(\infty, z) = \left( \frac{1}{\rho(z)} \right)^{1/2},
\end{align*}
where $\bv = (1, 1)^t$, and 
\begin{align*}
   \Sigma(z, w) = \bigl (I + M_1(z,w) \bigr ) \bigl (I - M_2(z,w) M_1(z,w) \bigr )^{-1} \bigl (I + M_2(z, w) \bigr ). 
\end{align*}
\end{thm}

The proof of Theorem \ref{thm:gen_formula} centers on evaluating the trace of mixed resolvents of self-adjoint $2 \times 2$ operator matrices whose entries are formed by $K_1, K_2, V_1,$ and $V_2$ (see Lemma \ref{lem:d_for}). The key strategy is to expand these resolvents into series of alternating products of simpler resolvents, each depending solely on either $\{K_1, V_1\}$ or $\{K_2, V_2\}$ (see Lemma \ref{lem:expension}). In this process, the analytic subordination theory for operator-valued free additive convolutions developed by Belinschi, Mai, and Speicher \cite{BMS15} plays a crucial role.

As an application, in Section 3, we evaluate this formula in two specific cases: first, when $K_1 = aI$ and $K_2 = bI$ for some scalars $a, b \in (0, 1)$; and second, when $Q(\infty)$, $Q(0)$, and $Q(-1)$ are freely independent in $\mathcal{M}$. In particular, we prove the following result.

\begin{thm}\label{thm:distance_free}
For the three freely independent projections $Q(0)$, $Q(-1)$, and $Q(\infty)$, the distance $d(z,w)$ is given by
\begin{align*}
 d(z, w) = \left( \frac{2|z-w|^2}{D(z,w)} \right)^{1/2}, \qquad \forall z, w \in \Cb, z \neq w,    
\end{align*}
where 
\begin{align*}
   D(z, w) &:= (|z| + |w|)(|z+1| + |w+1|) + |z-w|^2\\
   & \qquad + (|z||w+1| + |w||z+1|)(|z| + |w| + |z+1| + |w+1|).
\end{align*}
And 
\begin{align*}
    d(z, \infty) = d(\infty, z) =  \left( \frac{1}{1 + |z| + |z+1|} \right)^{1/2}.  
\end{align*}
\end{thm}

This explicit distance formula reveals that, in the free projection case, the lattice sphere lies on the sphere in $L^2(\mathcal{M})$ of radius $1/\sqrt{6}$ centered at
\begin{align*}
    \frac{1}{3}\bigl(Q(0) + Q(-1) + Q(\infty)\bigr).
\end{align*}
Furthermore, the infinitesimal Riemannian metric induced by $d$ has constant Gaussian curvature $1$ with conical singularities at $0, -1,$ and $\infty$, each having cone angle $\pi$. Combined with the existence and uniqueness results for spherical metrics with prescribed conical singularities (\cite{E04, T91, LT92}), our calculation demonstrates that the classical spherical surface obtained by gluing two spherical octants can be realized isometrically by the reflexive lattice generated by three freely independent projections of trace $\frac{1}{2}$.

\section{The proof of Theorem \ref{thm:gen_formula}}
A family of von Neumann subalgebras $\Cla_1, \ldots, \Cla_m \subset \Cla$ is called freely independent if
\begin{align*}
    \tau(A_1 \cdots A_n) = 0
\end{align*}
whenever $A_k \in \Cla_{i_k}$ satisfies $\tau(A_k) = 0$, with $i_k \ne i_{k+1}$ for all $1 \leq k \leq n-1$ (see, e.g., \cite{V85, VDN92}). A collection of operators is said to be freely independent if the von Neumann algebras they generate are freely independent.

Throughout the rest of the paper, we assume that $V_1, V_2$ are unitaries with 
\begin{align*}
  \tau(V_1) = \tau(V_2) = 0,  
\end{align*}
and that $K_1, K_2 \in \Cla$ are positive contractions such that $\ker(I-K_1) = \ker(I-K_2) = \{0\}$, and
\begin{align*}
    \ker \left( \sqrt{K_1(I-K_1)^{-1}}V_1 - \sqrt{K_2 (I-K_2)^{-1}}V_2 \right) = \{0\}.
\end{align*}

\begin{rem}
Recall that a unitary $V \in \Cla$ is called a Haar unitary if $\tau(V^n) = 0$ for all $n \in \mathbb{Z} \setminus \{0\}$. Let $V_1, V_2, U$ be three freely independent unitaries in $\Cla$ with $\tau(V_1) = \tau(V_2) = \tau(U) = 0$. Then $V_1U, V_2U$ are freely independent Haar unitaries. It is easy to see that replacing $V_1$ and $V_2$ by $V_1 U$ and $V_2 U$ does not alter the value of the distance $d(z, w)$. Consequently, although we will not rely on this fact in what follows, we may always assume that $V_1$ and $V_2$ are Haar unitaries.  
\end{rem}

To simplify notation in the following discussion, we introduce the shorthand
\begin{align*}
    Y_j & := \sqrt{K_j(I-K_j)^{-1}},  & X_j & := Y_jV_j,\\
    c_1(z) & := 1+z, & c_2(z) & :=-z. 
\end{align*}
Note that $Y_j$ is a positive densely defined closed operator affilated with $\Cla$. We use $\mu_j$ to denote the distribution of $Y_j$, i,e., 
\begin{align*}
 \tau(f(Y_j)) = \int f(t) d\mu_j(t),   
\end{align*}
for every bounded Borel function on $[0, \infty)$.

For every $z \in \Cb$, let 
\begin{align*}
 H_j(z) =
 \begin{pmatrix}
    0 & c_j(z)X_j  \\
    \overline{c_j(z)}X_j^* & 0
\end{pmatrix}
\end{align*}
and 
\begin{align*}
  H(z) &= H_1(z)+H_2(z), \\
  R(z) &= \bigl (iI-H(z) \bigr )^{-1}.
\end{align*}

\begin{lem}\label{lem:d_for}
For $z, w \in \Cb$, we have $d(z, \infty)^2 = \re \bigl ( \tau_2(iR(z)) \bigr )$, and 
\begin{align*}
d(z,w)^2 &= \re \bigl(\tau_2(iR(z)) - \tau_2(iR(w)^*) - 2\tau_2(R(z)^*R(w)) \bigr ).
\end{align*}
\end{lem}

\begin{proof}
Recall that $T(z) =c_1(z)X_1+c_2(z)X_2$ and 
\begin{align*}
    2Q(z) - I =
    \begin{pmatrix} 
        I-2(I+T(z)T(z)^*)^{-1} & 2T(z)(I+T(z)^*T(z))^{-1} \\ 
        2(I+T(z)^*T(z))^{-1}T(z)^* & 2(I+T(z)^*T(z))^{-1}-I 
    \end{pmatrix}
\end{align*}
is a self-adjoint unitary. Then
\begin{align*}
U(z) &:= \begin{pmatrix}
   I & 0\\
   0 & iI 
\end{pmatrix} \bigl ( 2Q(z) -I \bigr )
\begin{pmatrix}
   I & 0\\
   0 & iI 
\end{pmatrix}\\
    &= \begin{pmatrix} 
        I-2(I+T(z)T(z)^*)^{-1} & 2iT(z)(I+T(z)^*T(z))^{-1} \\ 
        2i(I+T(z)^*T(z))^{-1}T(z)^* & I- 2(I+T(z)^*T(z))^{-1}
    \end{pmatrix}\\
    &= I- 2(I + H(z)^2)^{-1} + 2iH(z)(I+H(z)^2)^{-1}\\
    & = I - 2(I + iH(z))^{-1} = 1- 2iR(z). 
\end{align*}
Since $U(z)$ is a unitary for every $z \in \Cb$, we obtain
\begin{align*}
\|Q(z) - Q(w)\|_2^2 &= \frac{1}{4}\|U(z) - U(w)\|_2^2  = \frac{1 - \re\tau_2(U(z)^*U(w))}{2} \\
&= \re \bigl(\tau_2(iR(z)) - \tau_2(iR(w)^*) - 2\tau_2(R(z)^*R(w)) \bigr ).
\end{align*}

Finally, note that
\begin{align*}
 d(z, \infty)^2 = \tau \left ((I+T(z)T(z)^*)^{-1} \right ) = \re \bigl ( \tau_2(iR(z)) \bigr ).
\end{align*}
\end{proof}

For the remainder of this section, let 
\begin{align*}
 \CE = \tau \otimes \id_{M_2(\Cb)} \colon \Clm \to I \otimes M_2(\Cb)   
\end{align*}
denote the entrywise conditional expectation. Then $(\Clm, \CE, I \otimes M_2(\Cb))$ is an operator-valued W$^*$-probability space. Since the von Neumann subalgebras $\Clw^*(K_1, V_1)$ and $\Clw^*(K_2, V_2)$ are freely independent in $(\Cla, \tau)$, the tensor product algebras $\Clw^*(K_1,V_1) \otimes M_2(\Cb)$ and $\Clw^*(K_2,V_2) \otimes M_2(\Cb)$ are free with amalgamation over $I \otimes M_2(\Cb)$ with respect to $\CE$ (see \cite{V95}). Explicitly, this means that for any alternating sequence of elements $X_k \in \Clw^*(K_{j_k}, V_{j_k}) \otimes M_2(\Cb)$ (with $j_k \neq j_{k+1}$) with $\CE(X_k) = 0$ for all $k$, the expectation of their product vanishes
\begin{align*}
\CE(X_1 X_2 \cdots X_n) = 0.
\end{align*}
Furthermore, $\CE$ extends to a orthogonal projection from $L^2(\Clm)$ onto the subspace $I \otimes M_2(\Cb)$.

\begin{lem}\label{lem:krho}
For every $z \in \Cb$, there exist unique positive real numbers $k_1(z), k_2(z), \rho(z)$ such that $\CE(iR(z))=\frac{1}{\rho(z)}I$,
\begin{align*}
 k_j(z) \geq 1, \qquad k_1(z) + k_2(z) = 1+\rho(z),
\end{align*}
and
\begin{align*}
    \int \frac{k_j(z)}{k_j(z)^2+|c_j(z)|^2 t^2} d\mu_j(t) = \frac1{\rho(z)}, \qquad j=1,2.
\end{align*}
Moreover, if $c_j(z)=0$, then $k_j(z) = \rho(z)$.
\end{lem}

\begin{proof}
Fix $z \in \Cb$. To simplify notation throughout the proof, we simply write $H_j$ and $c_j$ instead of $H_j(z)$ and $c_j(z)$.

We first show the existence of $k_1(z), k_2(z)$, and $\rho(z)$. For every $N > 0$, let $Y_{j, N}: = \min(t, N)(Y_j)$ and
\begin{align*}
    H_{j,N} : = 
    \begin{pmatrix}
       0 & c_jY_{j, N}V_j\\
       \overline{c_j} V_j^*Y_{j,N} & 0 
    \end{pmatrix}.
\end{align*}
Applying Theorem 2.2 in \cite{BMS15} (see also Theorem 3.1 in \cite{BYZ24}) to the pair $(H_{1, N}, H_{2, N})$, there exists a pair of Fr\'{e}chet analytic maps
\begin{align*}
    \omega_1, \omega_2 \colon \mathbb{H}^{+}(M_2(\Cb)) \to \mathbb{H}^{+}(M_2(\Cb))
\end{align*}
such that $\im \omega_j(B) \geq \im B$, and  
\begin{align}\label{equ:krho_I}
    \CE\bigl( (\omega_j(B) - H_{j,N})^{-1} \bigr) = \CE\bigl( (B - H_{1,N} - H_{2,N})^{-1} \bigr) = \bigl ( \omega_1(B) + \omega_2(B) - B \bigr)^{-1},
\end{align}
where
\begin{align*}
    \mathbb{H}^{+}(M_2(\Cb)) = \left\{ B \in M_2(\Cb) \colon \frac{B - B^*}{2i} > \varepsilon I \mbox{ for some $\varepsilon > 0$} \right\}.
\end{align*}
Moreover,
\begin{align*}
  \omega_1(B) = \lim_{n} f_{B}^{(n)}(A), \qquad \forall A \in \mathbb{H}^{+}(M_2(\Cb)),   
\end{align*}
where  
\begin{align*}
   f_B (C): = \bigl (h_{2} (h_{1}(C) + B)) + B \bigr ), \qquad h_j(C):= \CE \bigl ( (C - H_{j, N})^{-1}  \bigr)^{-1} -C,
\end{align*}
for any $C \in \mathbb{H}^{+}(M_2(\Cb))$. And same result hold for $\omega_2(B)$ with $f_B$ repalced with $C \mapsto \bigl (h_{1} (h_{2}(C) + B)) + B \bigr )$. 

We now show that $-i\omega_j(iI)$ is positive scalar for $j = 1, 2$. For $r > 0$, we have 
\begin{align*}
(irI-H_{j,N})^{-1}
&=-\begin{pmatrix}
ir \bigl (r^2 + |c_j|^2Y_{j,N}^2\bigr )^{-1}& c_j Y_{j,N}\bigl (r^2 + |c_j|^2 Y_{j,N}^2\bigr )^{-1}V_j\\
\overline{c_j}V_j^* Y_{j,N}\bigl (r^2 + |c_j|^2Y_{j,N}^2\bigr )^{-1} &ir V_j^* \bigl (r^2 + |c_j|^2Y_{j,N}^2\bigr )^{-1} V_j
\end{pmatrix}
\end{align*}
Then the freeness of $Y_{j,N}$ and $V_j$, together with the condition $\tau(V_j) = 0$, implies that
\begin{align*}
    \CE \bigl( (irI-H_{j,N})^{-1} \bigr ) = -iL_j(r) I 
\end{align*}
where 
\begin{align*}
L_j(r) = \int \frac{r}{r^2+|c_j|^2 \min(t, N)^2} d\mu_j(t).
\end{align*}
Therefore, we have 
\begin{align*}
   h_j(irI) + iI = i\left (\frac{1}{L_{j}(r)} - r + 1\right )I.  
\end{align*}
Note that
\begin{align*}
    \frac{1}{L_{j}(r)} - r + 1 > 0
\end{align*}
since $0 < L_j(r) \leq\frac{1}{r}$. Hence there exist $k_{1,N}, k_{2, N} \geq 1$ such that 
\begin{align*}
    \omega_1(iI) = ik_{1,N}I, \qquad \omega_2(iI) = ik_{2,N}I.
\end{align*} 
Let 
\begin{align*}
  \rho_{N}:= k_{1,N} + k_{2, N} -1.
\end{align*}
Then equations \eqref{equ:krho_I} yields
\begin{align*}
    \left ( \int \frac{k_{j,N}}{k_{j,N}^2+|c_j|^2 \min(t, N)^2} d\mu_j(t) \right ) I = \frac{1}{\rho_{N}} I = \CE \bigl (i (iI - H_{1,N} - H_{2,N})^{-1}  \bigr ).
\end{align*}

Since $H_{1, N} + H_{2, N} \to H_{1} + H_2$ in measure, 
\begin{align*}
   (iI - H_{1, N} - H_{2, N})^{-1} \to (iI - H_{1} - H_{2})^{-1}
\end{align*}  
in measure (see, for example, Lemma 2.8.4 in \cite{DPF20}). In particular, there exists $\rho \geq 1$ such that 
\begin{align*}
   \CE \bigl ( i(iI - H_{1} - H_{2})^{-1}  \bigr ) = \lim_{n \to \infty} \frac{1}{\rho_N}I = \frac{1}{\rho}I. 
\end{align*} 
Since $k_{j, N} \le \rho_N$, the sequence $\{k_{j, N}\}_{N=1}^\infty$ is bounded for each $j \in \{1, 2\}$. Therefore, there exists a subsequence along which $k_{1, N}$ and $k_{2, N}$ converge simultaneously to limits $k_1$ and $k_2$, respectively. Hence,
\begin{align*}
    k_1 + k_2 - 1 = \rho.
\end{align*}
Finally, by the Dominated Convergence Theorem,
\begin{align*}
    \int \frac{k_j}{k_j^2 + |c_j|^2 t^2} d\mu_j(t) = \frac{1}{\rho},
\end{align*}
and $k_j = \rho$ if $c_j = 0$.

To show the uniqueness of $k_1, k_2$, and $\rho$, note first that $\rho$ is uniquely determined by the condition $\CE\bigl(iR(z)\bigr) = \frac{1}{\rho}I$. Assume that $l_1, l_2$ are positive numbers satisfying 
\begin{align*}
   \int \frac{l_j}{l_j^2+|c_j|^2 t^2} d\mu_j(t) = \frac1{\rho}
\end{align*} 
and $l_1 + l_2 = 1 + \rho$. 

Assume that $l_1 \neq k_1$ and $l_2 \neq k_2$. Let 
\begin{align*}
  f(t) := \frac1{k_j^2+|c_j(z)|^2t^2}, \qquad g(t) := \frac1{l_j^2+|c_j(z)|^2t^2}.
\end{align*}
Note that
\begin{align*}
 \int f(t)g(t) d\mu_j(t) =\frac{1}{l_j^2-k_j^2} \left (\int f(t) d\mu_j(t) - \int g(t) d\mu_j(t) \right) = \frac{1}{\rho k_jl_j(k_j+l_j)}.   
\end{align*}
Since $\mu_j$ is a probability measure and $f, g$ are nonincreasing, 
\begin{align*}
&\int f(t)g(t) d\mu_j(t) -\left(\int f(t) d\mu_j(t)\right)\left(\int g(t) d\mu_j(t) \right)\\
&=\frac{1}{2} \int \int (f(t)-f(s))(g(t)-g(s)) d\mu_j(t)d\mu_j(s) \geq0.
\end{align*}
Therefore, 
\begin{align*}
    \frac{1}{\rho k_jl_j(k_j+l_j)} \geq \frac{1}{\rho^2 k_j l_j},
\end{align*}
or equivalently
\begin{align*}
    k_j + l_j \leq \rho.
\end{align*}
Summing over $j \in \{1, 2\}$ yields
\begin{align*}
   2(1 + \rho) =  k_1 + l_1 + k_2 + l_2 \leq 2\rho.
\end{align*}
which is impossible since $\rho > 0$. Therefore $k_1 = l_1$ and $k_2 = l_2$.
\end{proof}

Hereafter, for every $z \in \Cb$, let $k_1(z)$, $k_2(z)$ and $\rho(z)$ be the unique positive real numbers satisfying the conditions in Lemma \ref{lem:krho}. \\

For every $z \in \Cb$, let
\begin{align*}
   R_j(z) := \bigl (ik_j(z)I - H_j(z) \bigr)^{-1}, \qquad Z_j(z) := i\rho(z)R_j(z) - I, 
\end{align*}
and 
\begin{align}\label{equ:beta_def}
  \beta_j(z) = \frac{\rho(z)- k_j(z)}{k_j(z)}. 
\end{align}

\begin{lem}\label{lem:Z_cexp} 
For every $z \in \Cb$, $\CE(Z_j(z)) = 0$,  
\begin{align*}
&\CE(Z_j(z)^*Z_j(z)) = \beta_j(z) I,\\
&\|Z_j(z)\| \leq \max\{1,\beta_j(z)\}.
\end{align*}
In particular, $Z_j(z)=0$ if $c_j(z)=0$. 
\end{lem}

\begin{proof}
For a fixed $z \in \Cb$, as in the proof of Lemma \ref{lem:krho}, we suppress the argument $z$ in notations such as $Z_j(z)$ and $k_j(z)$.

By Lemma \ref{lem:krho}, 
\begin{align*}
   \CE(Z_j) = \CE(i\rho R_j-I) = 0. 
\end{align*}
Note 
\begin{align*}
 Z_j^*Z_j = \bigl ( (\rho-k_j)^2I+H_j^2 \bigr ) \bigl ( k_j^2I+H_j^2 \bigr)^{-1},
\end{align*}
since
\begin{align*}
Z_j = \bigl (i (\rho - k_j) I + H_j \bigr ) \bigl (ik_jI-H_j \bigr )^{-1}.
\end{align*}
Then 
\begin{align*}
    \|Z_j\|^2  \leq \sup_{t \ge 0} \frac{(\rho - k_j)^2 + t^2}{k_j^2 + t^2} = \max\{1, \beta_j^2 \}.
\end{align*}

By Lemma \ref{lem:krho}
\begin{align*}
 \CE(Z_j^*Z_j) &= I + \bigl( (\rho - k_j)^2 - k_j^2 \bigr) \left( \int \frac{d\mu_j(t)}{k_j^2+|c_j(z)|^2t^2}\right)I = \left(\frac{\rho - k_j}{k_j}\right)I.   
\end{align*}
If $c_j = 0$, then $\rho = k_j$, so $\CE(Z_j^*Z_j) = 0$, which implies $Z_j = 0$.
\end{proof}

Combining Lemma \ref{lem:Z_cexp} with fact that the von Neumann algebras $\Clw^*(K_1, V_1) \otimes M_2(\Cb)$ and $\Clw^*(K_2, V_2) \otimes M_2(\Cb)$ are free with amalgamation over $I \otimes M_2(\Cb)$, we immediately obtain the following result.

\begin{lem}\label{lem:Zprod_orth}
For any $z, w \in \Cb$, any $C \in I \otimes M_2(\Cb)$, and any two alternating products $Z_{j_1}(z)\cdots Z_{j_m}(z)$ and $Z_{l_1}(w)\cdots Z_{l_n}(w)$, we have
\begin{align*}
    \CE \bigl( (Z_{j_1}(z)\cdots Z_{j_m}(z))^* C (Z_{l_1}(w)\cdots Z_{l_n}(w)) \bigr) = 0
\end{align*}
unless $m = n$ and $j_s = l_s$ for all $s \in \{1, \dots, n\}$.
\end{lem}

\begin{lem}\label{lem:expension}
For every $z \in \Cb$, 
\begin{align}\label{equ:expension_0}
    R(z) = \frac{1}{i\rho(z)} \left( I+\sum_{n\geq1} \sum_{\substack{i_1,\ldots,i_n\in\{1,2\}\\i_m\ne i_{m+1}}} Z_{i_1}(z)\cdots Z_{i_n}(z)
    \right).
\end{align}
In particular, 
\begin{align*}
    \tau_2 \bigl (R(z)^*R(z) \bigr ) = \frac{1}{\rho(z)}.
\end{align*}
\end{lem}

\begin{proof}
By Lemma \ref{lem:krho} and Lemma \ref{lem:Z_cexp}
  \begin{align*}
       \left \|Z_{j_1}(z) \cdots Z_{j_n}(z) \right \|_2^2 = \left ( \prod_{l=1}^n \beta_{j_l}(z) \right ),
   \end{align*}
where $\beta_j(z)$ is defined in \eqref{equ:beta_def}. Note that 
\begin{align*}
    \beta_1(z)\beta_2(z) = \frac{(\rho(z)-k_1(z))(\rho(z)-k_2(z))}{k_1(z)k_2(z)} = 1 - \frac{(k_1(z)+k_2(z)-1)}{k_1(z)k_2(z)} < 1.
\end{align*}
Then 
\begin{align*}
    &\left \|I+\sum_{n\geq1} \sum_{\substack{i_1,\ldots,i_n\in\{1,2\}\\i_m\ne i_{m+1}}} Z_{i_1}(z)\cdots Z_{i_n}(z) \right \|_2^2\\
    &= 1+\sum_{m \geq 0} \left (\beta_1(z) + \beta_2(z) \right ) \bigl(\beta_1(z)\beta_2(z) \bigr)^m + 2\sum_{m \geq 1} \bigl(\beta_1(z)\beta_2(z) \bigr)^m\\
    &= \frac{(\beta_1(z)+1)(\beta_2(z)+1)}{1-\beta_1(z)\beta_2(z)} = \rho(z).
\end{align*}
Therefore 
\begin{align*}
   S(z) := I+\sum_{n\geq1} \sum_{\substack{i_1,\ldots,i_n\in\{1,2\}\\i_m\ne i_{m+1}}} Z_{i_1}(z)\cdots Z_{i_n}(z) 
\end{align*}
is a vector in $L^2(\Clm)$. Since multiplication by $Z_j(z)$ is continuous on $L^2(\Clm)$. Let 
\begin{align*}
   S_j(z) := \sum_{n\geq1} \sum_{\substack{i_1,\ldots,i_n\in\{1,2\}\\i_m\ne i_{m+1} \\ i_1 = j} } Z_{i_1}(z)\cdots Z_{i_n}(z),  \qquad j \in \{1, 2\}.
\end{align*}
In the following, we treat $S(z), S_1(z), S_2(z)$ as densely defined closed operators affiliated $\Clm$.

Note that 
\begin{align}\label{equ:lem:expension_I}
   S(z) = I + S_{1}(z) + S_{2}(z), \qquad  S_j(z) = Z_j(z)(S(z) - S_j(z)).
\end{align}
Recall that
\begin{align*}
  I+Z_j(z) = i \rho(z) R_j(z). 
\end{align*}
The identities \eqref{equ:lem:expension_I} in the $*$-algebra of operators affiliated with $\Clm$  imply
\begin{align*}
    S_j(z)= (I+Z_j(z))^{-1}Z_{j}(z)S(z) = \bigl (I -  \frac{1}{i\rho(z)}R_j(z)^{-1} \bigr )S(z).
\end{align*}
Therefore
\begin{align*}
    \frac{1}{i\rho(z)} \bigl (iI - H(z) \bigr ) S(z) = \left [ 
       \frac{1}{i\rho(z)}(R_1(z)^{-1} + R_2(z)^{-1}) - I 
   \right] S(z) = I. 
\end{align*}
Multiplication by the bounded inverse $R(z)$ of $iI-H(z)$ proves the equation \eqref{equ:expension_0}.
\end{proof}

For $j=1,2$ and $z \in \Cb$, define Borel functions on
$[0,\infty)$ by
\begin{align*}
f_{j,z}(t)&=\frac{\rho(z)k_j(z)}{k_j(z)^2+|c_j(z)|^2t^2}-1,\\
g_{j,z}(t)&=\frac{\rho(z)c_j(z)t}{k_j(z)^2+|c_j(z)|^2t^2}.
\end{align*}
Note that $f_{j,z} = g_{j,z} \equiv 0$ if $c_j(z)=0$. 

Let 
\begin{align*}
M_j(z,w) = \begin{pmatrix}
    A_j(z,w) & B_j(z,w)\\
    \overline{B_j(z,w)} & A_j(z,w)
\end{pmatrix},
\end{align*}
where 
\begin{align*}
 A_j(z,w) &=\int f_{j,z}(t)f_{j,w}(t) d\mu_j(t),\\
 B_j(z,w) &=\int g_{j,z}(t)\overline{g_{j,w}(t)} d\mu_j(t).
\end{align*}

\begin{lem}\label{lem:ZCZ}
For every $z, w \in \Cb$ and every diagonal scalar matrix 
\begin{align*}
 C=\diag(a, b) = \begin{pmatrix}
   aI & 0\\
   0 & bI 
\end{pmatrix},
\end{align*}
we have
\begin{align}\label{equ:ZCZ_I}
\CE(Z_j(z)^*CZ_j(w)) = \diag\bigl(M_j(z,w)(a, b)^{t}\bigr).
\end{align}
\end{lem}

\begin{proof}
Note that
\begin{align*}
Z_j(z) &=\begin{pmatrix}
f_{j,z}(Y_j) & -ig_{j,z}(Y_j) V_j\\
-iV_j^* g_{j,z}(Y_j)^* & V_j^* f_{j,z}(Y_j) V_j
\end{pmatrix}.
\end{align*}
Then
\small{
\begin{align*}
Z_j(z)^* C Z_j(w) = 
\begin{pmatrix}
 a f_{j,z}(Y_j) f_{j, w}(Y_j) + b g_{j,z}(Y_j)g_{j,w}(Y_j)^* & i \bigl (bg_{j,z}(Y_j)f_{j,w}(Y_j) -a f_{j,z}(Y_j)g_{j,w}(Y_j) \bigr )V_j\\
 iV_j^* \bigl (ag_{j,z}(Y_j)^* f_{j,w}(Y_j) -b f_{j,z}(Y_j)g_{j,w}(Y_j)^* \bigr ) & V_j^* \bigl (b f_{j,z}(Y_j) f_{j, w}(Y_j) + a g_{j,z}(Y_j)^*g_{j,w}(Y_j) \bigr ) V_j 
\end{pmatrix}.  
\end{align*}
}

The freeness of $Y_j$ and $V_j$, together with the condition $\tau(V_j) = 0$, implies that the traces of the off-diagonal entries vanish. Hence, \eqref{equ:ZCZ_I} holds by the definition of the matrix $M_j(z,w)$.
\end{proof}

\begin{lem}\label{lem:MM_leq_one}
   For $z, w \in \Cb$, we have
\begin{align*}
\|M_j(z,w)\| \leq\sqrt{\beta_j(z)\beta_j(w)}, \qquad  \|M_1(z,w) M_2(z,w)\| < 1.
\end{align*}
\end{lem}

\begin{proof}
For fixed $z, w$, we suppress the arguments $z, w$ in notations $M_j(z,w)$, $A_j(z, w)$, and $B(z, w)$.

Since the eigenvalues of $M_j$ are $A_j \pm |B_j|$, we have 
\begin{align*}
\|M_j\|=|A_j|+|B_j| &\leq
\left(\int f_{j,z}(t)^2 d\mu_j(t)\right)^{1/2} \left(\int f_{j,w}(t)^2 d\mu_j(t)\right)^{1/2}\\
&\quad +
\left(\int|g_{j,z}(t)|^2 d\mu_j(t) \right)^{1/2} \left(\int|g_{j,w}(t)|^2 d\mu_j(t)\right)^{1/2}\\
&\leq
\left(\int(f_{j,z}(t)^2+|g_{j,z}(t)|^2) d\mu_j\right)^{1/2} \left(\int(f_{j,w}(t)^2+|g_{j,w}(t)|^2) d\mu_j\right)^{1/2}\\
&=\sqrt{\beta_j(z)\beta_j(w)}.
\end{align*}
The last equality follows from the identity
\begin{equation*}
    \beta_j(z) = \int \bigl( f_{j,z}(t)^2 + |g_{j,z}(t)|^2 \bigr) \, d\mu_j,
\end{equation*}
which is obtained by combining Lemmas \ref{lem:Z_cexp} and \ref{lem:ZCZ}. Hence
\begin{align*}
\|M_1 M_2\| \leq \sqrt{\beta_1(z)\beta_2(z)\beta_1(w)\beta_2(w)} <1.
\end{align*}
\end{proof}

Having assembled the necessary ingredients, we are now ready to prove Theorem \ref{thm:gen_formula}.

\begin{proof}[\textbf{The proof of Theorem \ref{thm:gen_formula}}:]
By Lemmas \ref{lem:d_for} ans \ref{lem:krho}, it suffices to show that
\begin{align*}
   \re \tau_2 \bigl (R(z)^*R(w)\bigr ) =  \frac{\bv^t \Sigma(z,w)\bv}{2\rho(z)\rho(w)}, \qquad \forall z, w \in \Cb.
\end{align*}
To simplify notation, we write $M_j(z,w)$ simply as $M_j$ throughout the proof. 

By Lemmas \ref{lem:Zprod_orth} and \ref{lem:ZCZ}, the elements
\begin{align*}
  Z_{j_1}(z) \cdots Z_{j_m}(z) \quad\text{and}\quad Z_{l_1}(w) \cdots Z_{l_n}(w)
\end{align*}
are orthogonal as vectors in $L^2(\Clm)$ unless $m = n$ and $j_k = l_k$ for all $k = 1, \ldots, n$. And 
\begin{align*}
   \tau_2 \bigl (Z_{j_n}(z)^* \cdots Z_{j_1}(z)^* Z_{j_1}(w) \cdots Z_{j_n}(w) \bigr ) = \frac{1}{2}\mathbf{v}^t \bigl (M_{j_n} \cdots M_{j_1} \bigr ) \mathbf{v}
\end{align*}
Since there are exactly two alternating strings of each positive length, combining the expansion of $R(z)$ from Lemma \ref{lem:expension} with Lemma~\ref{lem:MM_leq_one} yields
\begin{align*}
\tau_2(R(z)^*R(w)) 
&= \frac{1}{2\rho(z)\rho(w)} \bv^t \left( I + \sum_{n=1}^\infty \sum_{\substack{j_1,\dots,j_n \in \{1,2\}\\ j_k \neq j_{k+1}}} M_{j_n} \cdots M_{j_1} \right) \bv \\
&= \frac{1}{2\rho(z)\rho(w)} \bv^t (I + M_1) \left( \sum_{m=0}^\infty (M_2 M_1)^m \right) (I + M_2) \bv \\
&= \frac{1}{2\rho(z)\rho(w)} \bv^t (I + M_1)(I - M_2 M_1)^{-1}(I + M_2) \bv.
\end{align*}
Moreover, because the adjoint $(M_{j_n} \cdots M_{j_1})^* = M_{j_1} \cdots M_{j_n}$ appears in the summation whenever $M_{j_n} \cdots M_{j_1}$ does, the matrix
\begin{align*}
    \Sigma(z,w) = (I + M_1)(I - M_2 M_1)^{-1}(I + M_2)
\end{align*}
is self-adjoint. Consequently,
\begin{align*}
\re \tau_2\bigl(R(z)^*R(w)\bigr) = \frac{\bv^t \Sigma(z,w)\bv}{2\rho(z)\rho(w)} \in \Rb.
\end{align*}
\end{proof}

\section{Examples}
As an application of Theorem \ref{thm:gen_formula}, we derive the distance formula in two extreme cases.

\subsection{The scalar case $K_1= aI$, $K_2 = bI$} 
In this subsection, we assume that $K_1 = aI$, $K_2 = bI$, where $0 < a, b < 1$.

\begin{lem}\label{lem:abkrho}
With the above notations. 
\begin{align}\label{equ:abkrho_I}
 \rho(z) &=\left[ 1+\frac{2a}{1-a}|1+z|^2+\frac{2b}{1-b}|z|^2 +\left(\frac{a}{1-a}|1+z|^2 -\frac{b}{1-b}|z|^2\right)^2 \right]^{1/2},
\end{align}
and
\begin{equation}\label{equ:abkrho_II}
    \begin{aligned}
    k_1(z) = \frac{1}{2} \left ( 1+ \rho(z) - \frac{a}{1-a}|z+1|^2 + \frac{b}{1-b}|z|^2 \right ),\\
    k_2(z) = \frac{1}{2} \left ( 1+ \rho(z) + \frac{a}{1-a}|z+1|^2 - \frac{b}{1-b}|z|^2\right ).        
    \end{aligned}
\end{equation}
\end{lem}

\begin{proof}
For fixed $z \in \Cb$, we suppress the argument $z$ in notations such as $k_j(z)$ and $\rho(z)$. Let
\begin{align*}
    p_1 := \frac{a}{1-a}|z+1|^2, \qquad p_2 := \frac{b}{1-b}|z|^2. 
\end{align*}

By Lemma \ref{lem:krho}, we have
\begin{align*}
\frac{k_1}{k_1^2+ p_1} = \frac{k_2}{k_2^2+ p_2} = \frac{1}{\rho}, \qquad k_1+k_2 = 1 +\rho.
\end{align*}
This implies that
\begin{align*}
    k_2 - k_1 = p_1 - p_2.
\end{align*}
Hence
\begin{align*}
   k_1 = \frac{1}{2}(1 + \rho - p_1 + p_2), \qquad k_2 = \frac{1}{2}(1 + \rho + p_1 -p_2) 
\end{align*}
Substituting the expression of $k_j$ into $\frac{k_j}{k_j^2+ p_j} = \frac{1}{\rho}$ and solving for $\rho$, we obtain
\begin{align*}
   \rho =   \sqrt{1 + 2p_1 + 2p_2 + (p_1-p_2)^2}.
\end{align*}    
\end{proof}

From the relations
\begin{align*}
    \rho(z)k_1(z) = k_1(z)^2 + \frac{a}{1-a}|1+z|^2, \quad \rho(z)k_2(z) = k_2(z)^2 + \frac{b}{1-b}|z|^2,
\end{align*}
a direct computation yields
\begin{align*}
M_1(z,w) &= \frac{a}{(1-a)k_1(z)k_1(w)}
\begin{pmatrix}
0 & (1+z)(1+\overline{w})\\
(1+\overline{z})(1+w) & 0
\end{pmatrix}, \\  
M_2(z,w) &= \frac{b}{(1-b)k_2(z)k_2(w)}
\begin{pmatrix}
0 & z\overline{w}\\
\overline{z} w & 0
\end{pmatrix}.
\end{align*}

\begin{prop}
With above assumption. Let 
\begin{align*}
   O_1(z,w) := \frac{a(1+z)(1+\overline{w})}{(1-a)k_1(z)k_1(w)}, \qquad  
   O_2(z,w) := \frac{b z\overline{w}}{(1-b)k_2(z)k_2(w)}.  
\end{align*}
Then
\begin{align*}
d(z,w) = 
\begin{cases}
    \left( \frac1{\rho(z)}+\frac1{\rho(w)}
-\frac{2}{\rho(z)\rho(w)}
\operatorname{Re}\left(
\frac{(1+\overline{O_1(z,w)})(1+O_2(z,w))}
{1-\overline{O_1(z,w)}O_2(z,w)}\right) \right)^{1/2} & z, w \in \Cb;\\  
      \frac{1}{\sqrt{\rho(z)}} & z \in \Cb, w=\infty,\\ 
\end{cases}
\end{align*}
where $k_1(z)$, $k_2(z)$, and $\rho(z)$ are given by \eqref{equ:abkrho_I} and \eqref{equ:abkrho_II}.
\end{prop}

\begin{proof}
For fixed $z, w \in \Cb$, we suppress the arguments $z, w$ in notations such as $O_1(z,w)$ and $O_2(z,w)$. Note that 
\begin{align*}
 M_2M_1 &= \diag \bigl(\overline{O_1}O_2, O_1\overline{O_2}\bigr).
\end{align*}
Then 
\begin{align*}
    \bv^t (I+ M_1) (I- M_2M_1)^{-1} (I+M_2) \bv = 
    2\re \left( \frac{(1+\overline{O_1})(1+O_2)}{1-\overline{O_1}O_2}\right),
\end{align*}
where $\bv = (1, 1)^t$. And the proposition is proved by substituting this identity into the distance formula of Theorem \ref{thm:gen_formula}.
\end{proof}

\subsection{The freely independent projection case} In this subsection, we assume that $Q(\infty)$, $Q(0)$, and $Q(-1)$ are freely independent projections with trace $1/2$ in $\Clm$. In particular, we assume that for $i = 1, 2$, each $K_i$ has the same distribution with respect to $\tau$ as $\cos^2(\frac{\pi}{2}\theta)$, where $\theta$ is uniformly distributed on $[0, 1]$, and that $V_1, V_2$ are Haar unitaries (see, for example, \cite[Section 3]{Y17}). \\

For any bounded Borel function $f$ on $(0,\infty)$,
\begin{align*}
\tau(f(Y_j)) &= \int_0^1 f\left(\sqrt{\frac{\cos^2(\frac\pi 2\theta)}{1-\cos^2(\frac\pi 2\theta)}}\right) d\theta \\
&= \int_0^1 f\left(\cot\left(\frac\pi 2\theta\right)\right)\,d\theta 
=\int_0^\infty f(t)\,\frac{2\,dt}{\pi(1+t^2)}. 
\end{align*}
It follows that the densely defined closed positive operator $Y_j$ follows the half-Cauchy law 
\begin{align*}
  d\mu(t) = \frac{2}{\pi(1+t^2)} dt,
\end{align*}
$t \in (0, \infty)$.

The following auxiliary identities involving the half-Cauchy distribution will be used in subsequent computations. For completeness, we briefly sketch their derivations below.

\begin{lem}\label{lem:aux_formula}
For $k>0$ and $c\geq0$,
\begin{align*}
\int \frac{d\mu(t)}{k^2+c^2t^2}&=\frac1{k(k+c)}.
\end{align*}
Moreover,
\begin{align*}
\int \frac{t^2\,d\mu(t)}{(k^2+c^2t^2)(l^2+d^2t^2)} &=\frac1{(k+c)(l+d)(cl+dk)},\\
\int \frac{d\mu(t)}{(k^2+c^2t^2)(l^2+d^2t^2)} &=\frac{cl+dk+cd}{kl(k+c)(l+d)(cl+dk)},
\end{align*}
for every $c,d,k,l>0$. 
\end{lem}

\begin{proof}
For $c = 0$, the first equation holds immediately. Assume that $k \neq c$, we have 
\begin{align*}
\int \frac{d\mu(t)}{k^2+c^2t^2} = \frac{2}{\pi(k^2-c^2)} \left( \int \frac{dt}{1+t^2} - \int \frac{c^2\,dt}{k^2+c^2t^2} \right) = \frac{1}{k(k+c)}.
\end{align*}
By the Dominated Convergence Theorem, continuity, the result also holds for $k = c$.

To prove the second equality, assume first that $cl \neq dk$. Applying the first formula, we have 
\begin{align*}
\int \frac{t^2\,d\mu(t)}{(k^2+c^2t^2)(l^2+d^2t^2)} 
&= \int \frac{1}{c^2l^2 - d^2k^2}\left( \frac{l^2}{l^2+d^2t^2} - \frac{k^2}{k^2+c^2t^2} \right) d\mu(t)\\
&= \frac{1}{(k+c)(l+d)(cl+dk)}.
\end{align*}
The identity extends to $cl = dk$ by continuity.

Finally, using the first two identities, 
\begin{align*}
\int \frac{d\mu(t)}{(k^2+c^2t^2)(l^2+d^2t^2)} 
&= \int \frac{1}{k^2}\left( \frac{1}{l^2+d^2t^2} - \frac{c^2t^2}{(k^2+c^2t^2)(l^2+d^2t^2)} \right) d \mu(t)\\
&= \frac{1}{k^2}\left( \frac{1}{l(l+d)} - \frac{c^2}{(k+c)(l+d)(cl+dk)} \right) \\
&=\frac{(cl+dk+cd)}{kl(k+c)(l+d)(cl+dk)}. 
\end{align*}
\end{proof}

By Lemma \ref{lem:aux_formula}, the functions
\begin{align*}
    k_1(z) = 1 + |z|, \quad k_2(z) = 1 + |z+1|, \quad, \rho(z) = 1 + |z| + |z+1|
\end{align*}
satisfy the conditions of Lemma \ref{lem:krho}. And for every $z, w \in \Cb$,  
\begin{align*}
 M_j(z,w) = 
\begin{cases}
0 & c_j(z) = c_j(w) = 0\\
 \frac{1}{l_j(z,w)}
\begin{pmatrix}
|c_j(z)||c_j(w)|&c_j(z)\overline{c_j(w)}\\
\overline{c_j(z)}c_j(w)&|c_j(z)||c_j(w)|
\end{pmatrix}   &  \mbox{otherwise}\\
\end{cases}.
\end{align*}
where 
\begin{align*}
l_j(z,w) = |c_j(z)|k_j(w)+|c_j(w)|k_j(z),
\end{align*}

Note that whenever $c_j(z)c_j(w) \neq 0$, $M_{j}(z, w)$ is a scalar multiple of a rank-one projection. Therefore, 
\begin{equation}\label{equ:RzRw_I}
    \begin{aligned}
        & M_1(z,w) M_2(z,w) M_1(z,w) = \kappa(z,w) M_1(z,w),\\
        & M_2(z,w) M_1(z,w) M_2(z,w)=\kappa(z,w) M_2(z,w),
    \end{aligned}
\end{equation}
where 
\begin{align*}
    \kappa(z,w) : = 2\tr \bigl (M_1(z,w)M_2(z,w) \bigr ).
\end{align*}

\begin{lem}\label{lem:kappa_expand}
For distinct $z, w \in \Cb$, we have 
\begin{align*}
  \kappa(z,w) = \frac{\bigl(|z||w+1| + |w||z+1|\bigr)^2 - |z-w|^2}{l_1(z,w)l_2(z,w)}.
\end{align*}
\end{lem}

\begin{proof}
Direct computation yields
 \begin{align*}
2\tr \bigl (M_1(z,w)M_2(z,w) \bigr ) 
= 
\frac{2|z||w||z+1||w+1| + 2 \re \bigl((1+z)(1+\bar{w})\bar{z}w\bigr)}{l_1(z,w)l_2(z,w)}.
\end{align*}   
Since $|(1+z)w - (1+w)z|^2 = |z-w|^2$, it follows that
\begin{align*}
2\re \bigl((1+z)(1+\bar{w})\bar{z}w\bigr) = |z|^2|w+1|^2 + |w|^2|z+1|^2 - |z-w|^2,
\end{align*}
which proves the identity in the lemma.
\end{proof}

\begin{lem}\label{lem:RzRw}
For $z, w \in \Cb$, $\tau_2 \bigl (R(z)^*R(w) \bigr )$ equals
\begin{align*}
 \frac1{2\rho(z)\rho(w)}\,
 \mathbf{v}^t\left[I+
 \frac{M_1(z, w) + M_2(z, w) + M_1(z, w)M_2(z, w) + M_2(z, w)M_1(z, w)}{1- \kappa(z,w)}\right]\mathbf{v},
\end{align*}
where $\mathbf{v}=(1,1)^{t}$. 
\end{lem}

\begin{proof}
To simplify notation, we write $M_j(z,w)$ and $\kappa(z,w)$ simply as $M_j$ and $\kappa$ throughout the proof.

By the proof of Theorem \ref{thm:gen_formula}, we have
\begin{align}\label{equ:RzRw_II}
   \tau_2(R(z)^*R(w)) &= \frac{1}{2\rho(z)\rho(w)} \bv^t \left( I + \sum_{n=1}^\infty \sum_{\substack{j_1,\dots,j_n \in \{1,2\}\\ j_k \neq j_{k+1}}} M_{j_n} \cdots M_{j_1} \right) \bv 
\end{align} 
Note that there are exactly two alternating strings of each positive length. By equation \eqref{equ:RzRw_I}, we have 
\begin{align*}
M_{j_n}\cdots M_{j_1}
&=\begin{cases}
\kappa^m M_1, & n=2m+1, j_1=1, \quad m\geq0,\\
\kappa^m M_2, &n=2m+1, j_1=2, \quad m\geq0,\\
\kappa^{m-1} M_2M_1, & n=2m, j_1=1,\quad m\geq1,\\
\kappa^{m-1} M_1M_2, & n=2m, j_1=2,\quad m\geq1.
\end{cases}
\end{align*}
Combining this with equation \ref{equ:RzRw_II}, we obtain
\begin{align*}
    \tau_2\bigl ( R(z)^* R(w) \bigr ) = \frac{1}{2\rho(z)\rho(w)}\mathbf{v}^t \left ( I+\frac{M_1+M_2+M_1M_2+M_2M_1}{1-\kappa} \right ) \mathbf{v}.
\end{align*}
\end{proof}

\begin{proof}[\textbf{The proof of Theorem \ref{thm:distance_free}}]
By Lemmas \ref{lem:d_for} and \ref{lem:krho}, it suffices to consider the case where $z, w$ are distinct numbers in $\Cb$. As in the proof of Lemma \ref{lem:RzRw}, we abbreviate $M_j(z,w)$, $l_j(z, w)$, and $\kappa(z,w)$ to $M_j$, $l_j$, and $\kappa$, respectively.

By Theorem \ref{thm:gen_formula} and Lemma \ref{lem:RzRw}, we have 
\begin{align}\label{equ:d_temp}
d(z,w)^2 =\frac1{\rho(z)}+\frac1{\rho(w)}
 -\frac1{\rho(z)\rho(w)}\mathbf v^t
 \left(I+\frac{M_1+M_2+M_1M_2+M_2M_1}{1-\kappa}\right)
 \mathbf v.
\end{align}
Since $z \neq w$, we have 
\begin{align*}
l_1 &=|z+1|+|w+1|+|z||w+1|+|w||z+1| > 0,\\
l_2 &=|z|+|w|+|z||w+1|+|w||z+1| > 0.
\end{align*}
Direct computation yields
\begin{align*}
&\mathbf{v}^t M_1 \mathbf{v} =\frac{2\bigl (|z+1||w+1|+\re((1+z)(1+\overline w)) \bigr)}{l_1}, \quad 
\mathbf{v}^t M_2 \mathbf{v} =\frac{2 \bigl (|z||w|+\re(z\overline w) \bigr )}{l_2},\\
&\mathbf{v}^t(M_1M_2+M_2M_1) \mathbf{v} = 2\kappa+\frac{4|z||w|\re \bigl ((1+z)(1+\overline w) \bigr)}{l_1l_2}+ \frac{4|z+1||w+1|\re(z\overline w)}{l_1l_2}.
\end{align*}

Using the identities
\begin{align*}
2\re(z\overline w) &=|z|^2+|w|^2-|z-w|^2,\\
2\re((1+z)(1+\overline w))
 &= |z+1|^2+|w+1|^2-|z-w|^2,
\end{align*}
and Lemma \ref{lem:kappa_expand}, we have 
\begin{align*}
    \mathbf{v}^{\mathsf{T}} \left( I + \frac{M_1 + M_2 + M_1 M_2 + M_2 M_1}{1 - \kappa} \right) \mathbf{v} 
    &= \frac{2}{l_1 l_2 (1 - \kappa)} \bigl [ l_1 l_2 + |z||w|l_1 + |z+1||w+1|l_2 \\
    &\quad + \bigl( l_1 + 2|z+1||w+1| \bigr) \re(z \overline{w}) \\
    &\quad + \bigl( l_2 + 2|z||w| \bigr) \re\bigl( (1+z)(1+\overline{w}) \bigr) \bigr ] \\
    &=  \rho(z) + \rho(w) - \frac{2|z-w|^2\rho(z)\rho(w)}{l_1l_2 (1- \kappa)}. 
\end{align*}
Substitution into equation \eqref{equ:d_temp}, we have 
\begin{align*}
d(z,w)^2 =\frac{2|z-w|^2}{l_1(z,w)l_2(z,w)(1-\kappa(z,w))}.
\end{align*}
Expanding the product $l_1 l_2 (1-\kappa)$ yields the explicit expression for the denominator:
\begin{align*}
l_1 l_2 (1-\kappa) &= (|z| + |w|)(|z+1| + |w+1|) + |z-w|^2 \\
&\quad + (|z||w+1| + |w||z+1|)(|z| + |w| + |z+1| + |w+1|).
\end{align*}
\end{proof}

\begin{rem}
Note that 
\begin{align*}
    d(z, 0)^2 &= \frac{|z|}{1 + |z| + |z+1|},\\
    d(z, -1)^2 &= \frac{|z+1|}{1 + |z| + |z+1|},\\
    d(z, \infty)^2 &=  \frac{1}{1 + |z| + |z+1|}.
\end{align*}
We obtain
\begin{align*}
 \|Q(z) - Q(0)\|_2^2 + \|Q(z) - Q(-1)\|_2^2 + \|Q(z) - Q(\infty)\|_2^2 = 1, \qquad \forall z \in \widehat{\Cb}.
\end{align*}
Since $\tau_2(Q(j)Q(k)) = \frac{1}{4}$ for distinct $k, j \in \{0, -1, \infty\}$, we have  
\begin{align*}
     \left \|Q(z) - \frac{1}{3} \bigl( Q(0) + Q(-1) + Q(\infty) \bigr ) \right \|_2^2 = \frac{1}{6}, \qquad \forall z \in \widehat{\Cb}.
\end{align*}
Therefore, $\{Q(z) : z \in \widehat{\Cb}\}$ is contained in the sphere of radius $\frac{1}{\sqrt{6}}$ centered at $\frac{1}{3} \big( Q(0) + Q(-1) + Q(\infty) \big)$ in $L^2(\Clm)$.
\end{rem}

For every $z \in \Cb \setminus \{0, -1\}$, 
\begin{align*}
    D(z,z) = 4|z||z+1|\bigl (1 + |z| + |z+1| \bigr) > 0.
\end{align*}
Then, for $\varepsilon \in \Cb$ near $0$, we have 
\begin{align*}
    d(z, z + \varepsilon)^2 = \frac{2|\varepsilon|^2}{D(z, z)} + o(|\varepsilon|^2).
\end{align*}
Thus, writing $z = x + iy$, the distance function $d(z, w)$ determines a conformal Riemannian metric $g$ on $\Cb \setminus \{0, -1\}$ with components
\begin{align}\label{equ:con_met}
    g_{xx} = g_{yy} = \frac{1}{2|z||z+1|(1 + |z| + |z+1|)}, \qquad g_{xy} = g_{yx} = 0.
\end{align}
And the Gaussian curvature 
\begin{align*}
    K_g(z) = 4(|z||z+1|(1 + |z| + |z+1|) \frac{\partial^2}{\partial_{z}\partial_{\overline{z}}} \log \bigl ((1 + |z| + |z+1|) \bigr) = 1.
\end{align*}
Near the singularity $z = 0$, the metric takes the form
\begin{equation*}
    g = \frac{1}{|z|} h(z) |dz|^2,
\end{equation*}
where 
\begin{equation*}
    h(z) = \frac{1}{2|z+1|(1 + |z| + |z+1|)}
\end{equation*}
is continuous and strictly positive near $0$, with $h(0) = \frac{1}{4}$. This shows that $z = 0$ is a conical singularity of cone angle $\pi$. An identical analysis shows that $z = -1$ and $z = \infty$ are also conical singularities of cone angle $\pi$. Therefore, $g$ is the unique conformal metric of constant Gaussian curvature $K = 1$ on the $2$-sphere with conical singularities of angle $\pi$ at $\{0, -1, \infty\}$ (see \cite{E04}; see also \cite{T91, LT92, MP16} for further background on conical metrics on $2$-spheres). More explicitly, let
\begin{align*}
    X(z) &:= \frac{|z| - |z+1| + 1}{1 + |z| + |z+1|}, \\
    Y(z) &:= \frac{|z+1| - |z| + 1}{1 + |z| + |z+1|}, \\
    Z(z) &:= \frac{|z| + |z+1| - 1}{1 + |z| + |z+1|}.
\end{align*}
The map
\begin{equation*}
    z \mapsto \bigl( \sqrt{X(z)}, \sqrt{Y(z)}, \sqrt{Z(z)} \bigr)
\end{equation*}
defines an isometic diffeomorphism from the upper (resp. lower) half-plane onto the intersection of the unit $2$-sphere with the positive open octant $\{(x, y, z) \in \mathbb{R}^3 : x, y, z > 0\}$.


\printbibliography

\end{document}